\documentclass[a4paper]{amsart}

\usepackage{amsmath,amssymb,amstext,amsthm,amscd,hyperref}
\usepackage{comment}

\newtheorem{theorem}{Theorem}[section]

\newtheorem{proposition}[theorem]{Proposition}

\newtheorem{lemma}[theorem]{Lemma}

\title[Automorphisms of $\mathbb{C}^2$ with a round Siegel domain]{Lacunary series, resonances, and automorphisms of $\mathbb{C}^2$ with a round Siegel domain}
\author{Fran\c cois Berteloot}
\address{Universit\'e Toulouse III, 
Institut Math\'ematique de Toulouse, B\^at. 1R2,
118 route de Narbonne,
F-31062 Toulouse Cedex 9, France. }
\email{francois.berteloot@math.univ-toulouse.fr} 

\author{Davoud Cheraghi}
\address{Department of Mathematics, 
Imperial College London, 
London, SW7 2AZ, United Kingdom}
\email{d.cheraghi@imperial.ac.uk}

\date{}


\begin{document}

\begin{abstract}
We construct transcendental automorphims of $\mathbb{C}^{2}$ having an unbounded and regular Siegel domain.
\end{abstract}

\maketitle

\section{introduction}
As shown by Cremer in 1927, a holomorphic germ $f(z)=e^{2\pi i \theta} z+ a_2 z^2+\cdots$, with irrational rotation number 
$\theta$, is not always linearisable at the origin \cite{Cre38}. 
However, Siegel proved in 1940 that linearisability occurs when $\theta$ is Diophantine \cite{Sie42}. Thirty years later  
Brjuno extended Siegel's result showing that the arithmetic condition $\sum_n q_n^{-1} \log q_{n+1} < + \infty$ 
on the best rational approximants $(p_n/q_n)_{n\geq 0}$ of $\theta$ is a sufficient condition for linearisability. 
In 1995, Yoccoz proved that if the Brjuno condition does not hold for some $\theta$, then the quadratic map 
$e^{2\pi i \theta} z+ z^2$ is not linearisable \cite{Yoc95}.

When $f$ is a rational function or entire, any such linearisable fixed point gives rise to a Siegel domain, i.e.\ a Fatou component of 
$f$ on which $f$ is holomorphically conjugate to the irrational rotation $z\mapsto e^{2\pi i \theta} z$. 
Boundaries of Siegel discs are usually very irregular \cite{Mc04,PZ04} but, Perez-Marco (unpublished manuscript 
\cite{Perez-Marco97}) and Avila-Buff-Cheritat \cite{ABC04} proved that Siegel domains with smooth boundaries do actually exist.

Siegel domains also exist in higher dimension and the canonical class of biholomorphic models for them is that of Reinhardt 
domains. 
This is a rather large class where one does not expect a Siegel domain to be holomorphically equivalent to a familiar 
Reinhardt domain, such as the Euclidean ball or the poly-disc. 
It is very likely that, generically, two distinct Siegel domains in $\mathbb{C}^{k\ge2}$ are not bi-holomorphic equivalent.
These questions remain widely open and very few explicit examples are known \cite{FornSibony1992,Bedford2018}. 
 
Here, we exhibit examples of shear automorphisms of $\mathbb{C}^{2}$ for which the product of the complex line 
with the Euclidean unit disc $\mathbb{C}\times\Delta$ is a Siegel domain. 
Indeed, the Siegel domain is the unique Fatou component of the map. 
These automorphisms are tangent to a rotation by an irrational number $R_\theta:=(e^{2\pi i \theta} w, e^{2\pi i \theta} z)$, 
where $\theta$ does not satisfy the Brjuno's arithmetic condition. 
This is in strong contrast with Yoccoz's theorem in dimension one, and is related to the resonance in the tangent map $R_\theta$. 

Assume that $\theta \in \mathbb{R} \setminus \mathbb{Q}$ with best rational approximants $(p_n/q_n)_{n\geq 0}$. 
Let $(q_n')_{n\geq 0}$ be a subsequence of $(q_n)_{n\geq 0}$, $\mu \in \mathbb{R} \setminus \mathbb{Q}$, 
and $u=(u_n)_{n\geq 0}$ belong to $l^\infty$, the space of bounded sequences in $\mathbb{C}$. 
Consider the formal power series 
\[\varphi_{\mu,q',u} (z):=ze^{2\pi i\mu}\sum_{n=0}^\infty u_{q'_{n}}(1-e^{2\pi i q'_n \mu}) z^{q'_n},\]
and let 
\[A_{\mu,q',u}(w,z) :=(e^{2\pi i \mu}w+ \varphi_{\mu,q',u}(z), e^{2\pi i \mu}z).\]

\begin{theorem}\label{T:Julia}
Assume that $\theta$ is an irrational number with best rational approximants $(p_n/q_n)_{n\geq 0}$ satisfying 
$\lim_{n\to \infty} q_n^{-1}\log q_{n+1}=\infty$. 
Let $S$ be  a countable dense subset  of $\mathbb{C}\times \{\vert z\vert >1\}$. 
Then, there exist a subsequence $q'$ of $(q_n)_{n\geq0}$ and a dense $G_\delta$-subset $E_2 \subset l^\infty$ 
such that for any $u \in E_2$, $A_{\mu,q',u}(w,z)$ is an automorphism of $\mathbb{C}^{2}$ satisfying the following properties:
\begin{itemize}
\item[i)]  $A_{\theta,q',u}$  is analytically conjugate to the rotation $R_\theta$ on $\mathbb{C}\times \Delta$;
\item[ii)]  the orbit of any $(w,z) \in \mathbb{C}\times \{\vert z\vert >1\}$ by $A_{\theta,q',u}$   is unbounded and recurrent;
\item[iii)] the sequence of derivatives $(\Vert \mathrm{D} A_{\theta,q',u}^{\circ N} (w,z)\Vert)_{N\geq 0}$ is unbounded, 
for any $(w,z)\in S$. 
\end{itemize}
\end{theorem}

\section{Proofs}
Let $\theta$ be an irrational real number, with continued fraction expansion 
\[\theta=a_0+ \cfrac{1}{a_1+ \cfrac{1}{a_2+ \cfrac{1}{a_3+ \dots} } },\] 
with integers $a_i\geq 1$, for all $i \geq 1$, and $a_0\in \mathbb{Z}$.  
The \textit{best rational approximants} of $\theta$ are defined as the sequence of rational numbers 
\[\frac{p_n}{q_n}= a_0+\cfrac{1}{a_1+ \cfrac{1}{a_2+ \ddots + \cfrac{1}{a_n}}},  \quad n\geq 0.\]
We shall assume that $p_n$ and $q_n$ are relatively prime, and $q_n \geq 0$, for all $n\geq 0$. 
It is known that the sequence of integers $p_n$ and $q_n$, for $n\geq 0$, may be defined according to the following recursive 
relations 
\begin{gather*}
p_{-2}=q_{-1}=0 , \qquad p_{-1}=q_{-2}=1, \\
q_{n+1}= a_{n+1} q_n+ q_{n-1}, \qquad p_{n+1}= a_{n+1} p_n+ p_{n-1}.
\end{gather*}
The sequence $(q_n)_{n\geq 0}$ has exponential growth. For instance, since $a_i \geq 1$, the sequence $q_n$ grows at least 
as fast as the Fibonacci sequence. Indeed, by an inductive argument, one may see that 
\begin{equation}\label{E:exponential-growth}
q_n \geq 2^{(n-1)/2}, \quad \forall n\geq 0. 
\end{equation}
Moreover, one has the following estimate on the size of approximation
\begin{equation}\label{E:fine-approximation}
\left| \theta - \frac{p_n}{q_n}  \right |  \leq \frac{1}{q_n q_{n+1}}, \quad \forall n \geq 0. 
\end{equation}
One may consult \cite{Khin64} or \cite{Schmidt1980} for basic properties of continued fractions. 

Given $z_0 \in \mathbb{C}$ and $r > 0$, we define 
\[\Delta(z_0, r):=\{z\in \mathbb{C}  \;\colon\; |z-z_0|<r\}, \quad \Delta:= \Delta(0,1).\] 

Let us consider the (complex) Banach space $(l^\infty,\Vert\;\Vert_\infty)$, where 
$\Vert (u_n)_{n\geq 0}\Vert_\infty$ is defined as $\sup_{n\geq 0} \vert u_n\vert$. 
We assume $u_n \in \mathbb{C}$, for $n\geq 0$. 

\begin{proposition}\label{PropRadii}
Let $\theta$ be an irrational number with best rational approximants $(p_n/q_n)_{n\geq 0}$. 
For every subsequence $q':=(q'_n)_{n\geq 0}$ of $(q_n)_{n\geq 0}$, the following hold:  
\begin{itemize}
\item[(i)] the Taylor series 
\[ h_{q'}(z):=z\sum_{n=0}^\infty z^{q'_{n}}\]
is holomorphic on $\Delta$, but does not have holomorphic extension across any $z \in \partial \Delta$;
\item[(ii)] for any $u=(u_n)_{n\geq 0} \in l^\infty$ the Taylor series 
\[ h_{q',u}(z):=z\sum_{n=0}^\infty u_{q'_{n}}z^{q'_{n}}\]
is holomorphic on $\Delta$;
\item[(iii)] there exists a dense $G_\delta$ subset $E_1$ of $l^\infty$ (which only depends on $q'$) such that 
for every $u \in E_1$, $ h_{q',u}(z)$ does not have holomorphic extension across any $z \in \partial \Delta$. 
\end{itemize}
\end{proposition}


The series in the above proposition are special cases of Lacunary series, studied first by Weierstrass; 
see \cite{Kahane64} or \cite{ErdMac1954}. 

\begin{proof}
By Equation~\eqref{E:exponential-growth}, the series $h_{q',u}$ satisfies Hadamard's lacunarity 
condition \cite{Hadamard1892}, 
that is, $\liminf_{n\geq 0} (q'_{n+1}/q'_n) > 1$. 
This implies that, if the radius of convergence of the series $h_{q',u}(z)$ about $0$ is equal to some $r>0$ then 
$h_{q',u}(z)$ does not extend holomorphically across any $z$ in $\partial \Delta(0,r)$. 

Evidently, the radius of convergence of $h_{q'}$ is equal to $+1$. 
The radius of convergence of $h_{q',u}(z)$ is at least $+1$. 
Below we show that the radius of convergence is indeed equal to $+1$ for many series of the form $h_{q',u}(z)$.  

For $k\geq 1$, consider the set 
\[E^k:=\left \{(u_n)_{n\geq 0}\in l^\infty \;  \colon \; \exists N\in\mathbb{N}, \Big | \displaystyle{\sum}_{n=0}^N u_{q_n'} \Big| >k \right\}.\]
For every $k \in \mathbb{N} $, $E^k$ is open and dense in $l^\infty$. 
By Baire's theorem, $E_1=\cap_{k \in \mathbb{N}} E^k$ is a dense subset of $l^\infty$. 
For every $u\in E_1$, the series $\sum_{n\geq 0} u_{q_n'}$ is divergent. 
\end{proof}

Let $h_{q',u}:\Delta \to \mathbb{C}$ be a map as in Proposition~\ref{PropRadii}, with $u\in E_1$.  
We may consider the shear map 
\[S_{h_{q',u}}(w,z)= (w+h_{q',u}(z), z).\]
This is an automorphism of $\mathbb{C} \times \Delta$, with inverse $S_{h_{q',u}}^{-1}= S_{-h_{q',u}}$. 
For $\mu \in \mathbb{R}$, let us consider the two dimensional rotation $R_\mu: \mathbb{C}^2 \to \mathbb{C}^2$,  
\[R_\mu (w,z)= (e^{2\pi i \mu}w, e^{2\pi i \mu}z).\]

We define the automorphism $A_{\mu,q',u}: \mathbb{C} \times \Delta \to \mathbb{C} \times \Delta$ as 
\begin{equation}\label{E:conjugacy}
A_{\mu,q',u}=  S_{h_{q',u}}^{-1} \circ R_\mu  \circ  S_{h_{q',u}}= S_{-h_{q',u}} \circ R_\mu  \circ  S_{h_{q',u}}.
\end{equation}
We note that 
\begin{align*}
A_{\mu,q',u}(w,z) & = S_{-h_{q',u}} \circ R_\mu  \circ  S_{h_{q',u}}(w,z) \\
&= S_{-h_{q',u}} \circ R_\mu (w+ h_{q',u}(z),z)\\
&= S_{-h_{q',u}} (e^{2\pi i \mu} w+ e^{2\pi i \mu} h_{q',u}(z), e^{2\pi i \mu} z ) \\
&= \left(e^{2\pi i \mu} w+ e^{2\pi i \mu} h_{q',u}(z)- h_{q',u}(e^{2\pi i \mu} z) , e^{2\pi i \mu} z \right).
\end{align*}

Let us define the map $\varphi_{\mu,q',u}: \Delta \to \mathbb{C}$ as 
\[\varphi_{\mu,q',u}(z)= e^{2\pi i \mu} h_{q',u}(z)- h_{q',u}(e^{2\pi i \mu} z),\]
so that 
\[A_{\mu,q',u}(w,z) =(e^{2\pi i \mu}w+ \varphi_{\mu,q',u}(z), e^{2\pi i \mu}z).\]
By Equation~\eqref{E:conjugacy}, for all $N \in \mathbb{N}$, 
\begin{equation}\label{E:composition}
A_{\mu,q',u}^{\circ N}(w,z) =A_{N\mu,q',u}(w,z)=(e^{2\pi i N\mu}w+ \varphi_{N\mu,q',u}(z), e^{2\pi i N\mu}z).
\end{equation}
When the sequence $u \in l^\infty$ is the constant sequence $u_n=1$ we shall simply denote the above maps by 
$A_{\mu,q'}$ and $\varphi_{\mu,q'}$, that is, 
\[\varphi_{\mu,q'}(z)=e^{2\pi i \mu} h_{q'}(z)- h_{q'}(e^{2\pi i \mu} z),\]
and 
\[A_{\mu,q'}(w,z) =(e^{2\pi i \mu}w+ \varphi_{\mu,q'}(z), e^{2\pi i \mu}z).\]

Evidently, 
\[\varphi_{\mu,q',u} (z)=ze^{2\pi i\mu}\sum_{n=0}^\infty u_{q'_{n}}(1-e^{2\pi i q'_n \mu}) z^{q'_n}.\]
\textit{A priori}, the above series is only defined and holomorphic on the unit disk $\Delta$. 
However, when $\mu=\theta$ and for special choices of $\theta$, the function
$\varphi_{\theta,q',u}$ becomes entire. This phenomenon is crucial in our study.

\begin{proposition}\label{P:growth-extra}
Assume that the best rational approximants $(p_n/q_n)_{n\geq 0}$ of $\theta \in \mathbb{R} \setminus \mathbb{Q}$ 
satisfy $q_n^{-1} \log q_{n+1} \to \infty$, as $n\to \infty$. 
Then, for any subsequence $q'$ of $(q_n)_{n\geq 0}$ and any $u\in l^\infty$, the Taylor series of 
$\varphi_{\theta,q',u}$ defines an entire holomorphic map on $\mathbb{C}$. 

In particular, $A_{\theta, q', u}$ is an automorphism of $\mathbb{C}^2$. 
\end{proposition}

By choosing $a_n$ large enough, one may identify values of $\theta$ such that $q_n$ grows arbitrarily fast. 
In particular, there are irrational numbers $\theta$ for which $q_n^{-1} \log q_{n+1} \to +\infty$ as $n \to +\infty$. 

\begin{proof}
By the estimate in Equation~\eqref{E:fine-approximation}, we have $|q_n \theta - p_n| \leq 1/q_{n+1}$. 
On the other hand, for any $\beta \in [-1/2,1/2]$, $|1-e^{2\pi i \beta}| \leq 2\pi |\beta|$. 
In particular,  $|1-e^{2\pi i q_n \theta} | \leq 2\pi /q_{n+1}$. 
By the hypothesis of the proposition on the growth of $q_n$, we conclude that the radius of convergence of the 
Taylor series for $\varphi_{\theta,q,u}$ is $\infty$. The Taylor series of $\varphi_{\theta,q',u}$ is a subseries
of the Taylor series of $\varphi_{\theta,q,u}$, and hence, its radius of convergence is also $\infty$.
\end{proof}

We shall also need to identify some $\mu\in \mathbb{R}$ for which the convergence radius of the Taylor series of $\varphi_{\mu,q',u}$
is precisely equal to $1$.

\begin{lemma}\label{LemMu}
Let $\theta \in \mathbb{R} \setminus \mathbb{Q}$ with best rational approximants $(p_n/q_n)_{n\geq 0}$ and let  
$q':=(q'_n)_{n\geq 0}$ be a subsequence of $(q_n)_{n\geq 0}$. 
There exist $\mu\in \mathbb{R}$ and a dense $G_\delta$-subset $E_1$ of $l^\infty$ such that 
for any $u \in E_1$, the convergence radii of the Taylor series of $\varphi_{\mu,q',u}$ and $\varphi_{\mu,q'}$ are equal to $1$.
\end{lemma}

\begin{proof}
Let $(q_n'')_{n\geq 0}$ be a subsequence of $q'$ such that for all $n\geq 0$ we have $q_{n+1}'' \geq 4 q_n''$.  
We may inductively choose a nest of closed connected intervals $I_n$ on $[0,1]$, for $n \geq 0$, 
such that each interval $I_n$ has length $1/q''_n$ and each set $q''_n I_{n+1}$ is contained in $[1/4,3/4] + \mathbb{Z}$. 
We define $\mu$ as the unique element in the nest $I_n$. 
It follows that for all $n\geq 0$, $|1- e^{2\pi i q''_n \mu}| \geq \sqrt{2}$.

According to Proposition \ref{PropRadii}, there exists a dense $G_\delta$-subset $E_1$ of $l^\infty$ such that 
for any $u\in E_1$ the radius of convergence of $h_{q'',u}(z)=z\sum_{n=0}^\infty u_{q''_{n}}z^{q''_{n}}$ is equal to $1$.
By Cauchy-Hadamard's formula, we must have $\limsup_n \vert u_{q''_{n}}\vert^{1/q''_n}=1$. 
Since $2\ge |1- e^{2\pi i q''_{n} \mu}| \geq \sqrt{2}$, it follows that for $u\in E_1$, 
\[\limsup_n \vert u_{q''_{n}} (1- e^{2\pi i q''_{n} \mu} )\vert^{1/q''_n}=1.\] 
As $q''$ is a subsequence of $q'$, and $\limsup_n   \vert u_{q'_{n}} (1- e^{2\pi i q'_{n} \mu} )\vert^{1/q'_n} \le 1$, 
we conclude that 
\[\limsup_n   \vert u_{q'_{n}} (1- e^{2\pi i q'_{n} \mu} )\vert^{1/q'_n} = 1.\] 
Therefore, for $u \in E_1$, the radius of convergence of $\varphi_{\mu,q',u} (z)$ is equal to $1$.

The above arguments also work when the sequence $u$ is constant equal to $1$ and thus the convergence
radius of $\varphi_{\mu,q'} (z)$ equals $1$ as well.
\end{proof}

Assume that $\theta$ is an irrational number which satisfies the hypothesis of Proposition~\ref{P:growth-extra}. 
For any subsequence $q'$ of $q$ and any $u\in l^\infty$, 
\[A_{\theta,q',u} (w,z)= (e^{2\pi i \theta} w+ \varphi_{\theta,q',u}(z), e^{2\pi i \theta}z)\]
is an automorphism of $\mathbb{C}^2$. 
The map $A_{\theta,q',u}$ has a constant Jacobian of size $1$. So, it is area preserving.
Moreover, $A_{\theta,q',u}$ preserves the origin, with derivative 
\[\mathrm{D}A_{\theta,q',u}(0,0)= \left(
\begin{array}{cc}
e^{2\pi i \theta} & 0 \\
0 & e^{2\pi i \theta}
\end{array}
\right).\]

Any irrational number $\theta$ satisfying $\lim_{n\to \infty} q_n^{-1} \log q_{n+1}=\infty$ does not satisfy the Brjuno's 
arithmetic condition for the linearisability \cite{Sie42,Brj71,Her87-2,Yoc95}.
However, here $D A_{\theta,q',u}(0,0)$ enjoys a resonance condition. 
By virtue of the commutative relation in Equation~\eqref{E:conjugacy}, $A_{\theta,q',u}$ is conjugate to the rotation $R_\theta$ 
on the round domain $\mathbb{C} \times \Delta$, via the map $S_{h_{q',u}}$. 
We would like to understand the global dynamics of $A_{\theta, q', u}$. In particular, if the Siegel disk of $A_{\theta, q',u}$ 
centred at $(0,0)$ extends beyond $\mathbb{C} \times \Delta$. 

Let $F: \mathbb{C}^n \to \mathbb{C}^n$ be a holomorphic map. We say that $\Omega$ is a Fatou component for the iterates 
of $F$, if the family $\{F^{\circ n}\}_{n\geq 0}$ forms a locally equi-continuous family of maps near any point in $\Omega$. 
One may consult \cite{Mi06,CG93,MNTU2000} for classic books in complex dynamics in dimension one, 
and refer to \cite{CGSY03,Sibony1999} for general theory of complex dynamics in higher dimensions. 

\begin{theorem}\label{T:Siegel-component}
Assume that $\theta$ is an irrational number whose best rational approximants $(p_n/q_n)_{n\geq 0}$ satisfies 
$\lim_{n\to \infty} q_n^{-1}\log q_{n+1}=\infty$.
Let $q'$ be a subsequence of $(q_n)_{n\geq 0}$. 
There exists a dense $G_\delta$-subset $E_1$ of $l^\infty$ such that for every $u\in E_1$, 
$\mathbb{C} \times \Delta$ is a Fatou component of $A_{\theta,q', u}$, and for every $(w,z) \in \mathbb{C}^2$ with $|z|>1$ 
the orbit of $(w,z)$ under $A_{\theta,q', u}$ is unbounded.
 
Moreover, the same conclusions hold for $A_{\theta,q'}$, that is, for the constant sequence $u_n\equiv 1$.  
\end{theorem}

\begin{proof}
Let us fix $q'$, and let $\mu \in \mathbb{R}$ and $E_1 \subset l^\infty$ be the $G_\delta$ dense set obtained in Lemma \ref{LemMu}. 
By Proposition~\ref{P:growth-extra}, for every $u\in l^\infty$, $A_{\theta, q', u}$ is a well-define holomorphic map on $\mathbb{C}^2$. 
From now on we assume that $u \in E_1$, or $u$ is the constant sequence $u_n \equiv 1$. 

Note that $A_{\theta,q',u}$ preserves $\mathbb{C} \times \Delta$. Indeed, by Equation~\ref{E:conjugacy} and 
Proposition~\ref{PropRadii}, $A_{\theta,q',u}$ is conjugate to an irrational rotation on $\mathbb{C} \times \Delta$. 
Thus, $\mathbb{C} \times \Delta$ is contained in the Fatou set of $A_{\theta,q',u}$. 
Let $\Omega \subset \mathbb{C}^2$ be a Fatou component of $A_{\theta,q',u}$ which contains $\mathbb{C} \times \Delta$. 
Assume in the contrary that $\Omega \neq \mathbb{C} \times \Delta$. 

We may choose $w_0\in \mathbb{C}$, $z_0 \in \partial \Delta$, and $\epsilon >0$ such that the bidisk 
$\Delta(w_0, \epsilon) \times \Delta(z_0, \epsilon)$ is contained in $\Omega$. It follows from the formula 
\begin{align*}
A_{\theta,q',u}^{\circ N} (w,z) =(e^{2\pi i N\theta}w+ \varphi_{N\theta,q',u}(z), e^{2\pi i N\theta}z), \quad N\geq 0,
\end{align*}
that the iterates of $A_{\theta,q',u}$ are locally equi-continuous on the tube $\mathbb{C} \times \Delta(z_0, \epsilon)$. 
On the other hand, since $\Omega$ is invariant under $A_{\theta,q',u}$, and $A_{\theta,q',u}$ acts as an irrational rotation in 
the $z$-coordinate, $\Omega$ contains $\mathbb{C} \times \{z\in \mathbb{C} \;\colon\; ||z|-1|<\epsilon\}$.
Thus, the iterates of $A_{\theta, q', u}$ are equi-continuous on the tube $T= \mathbb{C} \times \Delta(0, 1+ \epsilon)$. 

Since $A_{\theta,q',u}$ is conjugate to the irrational rotation $R_\theta$ on $\mathbb{C} \times \Delta$, there is an increasing 
sequence of positive integers $m_i$ such that 
\[A_{\theta,q',u}	^{\circ m_i} \to S_{-h_{q',u}} \circ R_\mu \circ S_{h_{q',u}},\]
with the convergence uniform on compact subsets of $\mathbb{C} \times \Delta$. 
Since $\Omega$ is a Fatou set for $A_{\theta,q',u}$, we may extract an increasing subsequence $k_i$ of $m_i$ such that the 
iterates $A_{\theta,q',u}^{\circ k_i}$ converges uniformly on compact subsets of  $T$ to a holomorphic map 
$\psi: T \to \mathbb{C}^2$. 
Thus, we must have 
\[\psi=S_{-h_{q',u}} \circ R_\mu \circ S_{h_{q',u}}\;\textrm{on}\;\mathbb{C} \times \Delta.\]
However, the map $S_{-h_{q',u}} \circ R_\mu \circ S_{h_{q',u}}$ is of the form 
\[(w,z) \mapsto (e^{2\pi i \mu}w+ \varphi_{\mu,q',u}(z), e^{2\pi i \mu}z),\]
which, by our choice of $\mu$ and $u$ and according to Lemma \ref{LemMu}, diverges at some point on 
$\mathbb{C} \times \partial \Delta$. This contradicts $\psi$ being defined on $T$.

We have just proved  that $\mathbb{C} \times \Delta$ is a Fatou component of $A_{\theta,q',u}$. 
As we shall show below, this implies that for any $(w,z) \in \mathbb{C}^2$ with $|z|>1$, the orbit of 
$(w,z)$ by $A_{\theta,q',u}$ is unbounded.

Let us fix $(w_0, z_0) \in \mathbb{C}^2$ with $|z_0| > 1$. 
Assume in the contrary that there is $M>0$ such that  $|A_{\theta,q',u}^{\circ N}(w_0, z_0)| \leq M$, for all $N\geq 0$.
Then $|A_{\theta,q',u}^{\circ (N+n)}(w_0, z_0)| \leq M$, for all $N\ge 0$ and all $n\ge 0$.
Writing $A_{\theta,q',u}^{\circ n} (w_0,z_0)=:(w_n,e^{2\pi i  n \theta}z_0)$, we thus have simultaneously
$\vert w_n\vert \le M$ and $|A_{\theta,q',u}^{\circ N} (w_n, e^{2\pi i n\theta}z_0)| \leq M$, for all $n \geq 0$ and all $N\geq 0$.
Then, since 
\begin{align*}
A_{\theta,q',u}^{\circ (N+n)}(w_0, z_0) &=A_{\theta,q',u}^{\circ N} (w_n,e^{2\pi i n\theta}z_0) \\
&=(e^{2\pi i N\theta}w_n+ \varphi_{N\theta,q',u}(e^{2\pi i n\theta}z_0), e^{2\pi i (N+n)\theta}z_0),
\end{align*}
we get that $\vert \varphi_{N\theta,q',u}\vert \le 2M$ on $\{e^{2\pi i n \theta} z_0 \;\colon\;  n\geq 0\}$, for any $N\ge 0$. 
Thus, the family of maps $\varphi_{N\theta,q',u}$, for $N\geq 0$, is uniformly bounded on the circle $|z|=|z_0|$ and, 
by the maximum principle, on the disk $\Delta(0, |z_0|)$. 

As $A_{\theta,q',u}^{\circ N} (w,z) =(e^{2\pi i N\theta}w+ \varphi_{N\theta,q',u}(z), e^{2\pi i N\theta}z)$,
by the above paragraph, the iterates of $A_{\theta,q',u}$ are locally uniformly bounded on 
$\mathbb{C} \times \Delta(0, |z_0|)$. Therefore, $\mathbb{C} \times \Delta(0, |z_0|)$ must be a Fatou component of 
$A_{\theta,q',u}$, which strictly contains the Siegel tube $\mathbb{C} \times \Delta$. This contradicts the earlier statement 
that $\mathbb{C} \times \Delta$ itself is a Fatou component of $A_{\theta,q',u}$.
\end{proof}

We now aim to show that for a certain choice of the subsequence $q'$ and for a generic  $u\in l^\infty$ the Julia set
of $A_{\theta,q',u}$ coincides with the complement of the Siegel  tube $\mathbb{C} \times \Delta$. This will prove
Theorem \ref{T:Julia}.

The proof  is based on the following two lemmas.

\begin{lemma}\label{LemRec}
Assume that $\theta$ is an irrational number whose best rational approximants $(p_n/q_n)_{n\geq 0}$ satisfies 
$q_n^{-1} \log q_{n+1} \to \infty$, as $n\to \infty$. 
For any sequence of positive numbers $(\varepsilon_p)_{p\geq 1}$ converging to zero, 
there exist a subsequence $(q'_n)_{n\geq 0}$ of $(q_n)_{n\geq 0}$ and an increasing  sequence of integers $(N_p)_{p \geq 1}$ 
such that for every $p \ge 1$, we have 
\[\vert 1-e^{2 \pi i N_{p}\theta} \vert + \sum_{n\ge 0} \vert 1-e^{2\pi i q'_{n} N_{p}\theta} \vert  p^{q'_n} \le \varepsilon_p.\]
\end{lemma}

\begin{proof} 
Let us set $N_0=1$ and $q'_0=q_0$. 
We inductively identify the pair of integers $q_p'$ and $N_p$ so that the following properties are satisfied for all $p\ge 1$:
\begin{itemize}
\item[(i)] $N_p > N_{p-1}$, 
\item[(ii)] 
$\vert 1-e^{2i\pi  N_{p}\theta} \vert + \sum_{n=0}^{p-1} \vert 1-e^{2i\pi q'_{n} N_{p}\theta} \vert  p^{q'_n} \le \varepsilon_p/2$,
\item[(iii)] $q'_p \in \{q_n \;\colon\;  n\geq 0, q_n >q'_{p-1}, q_n\geq 2 N_p\}$, 
\item[(iv)]  $\sum_{n ; q_n\ge q'_p} \frac{p^{q_n}}{q_{n+1}} \le \varepsilon_p/(4\pi N_p)$.
\end{itemize}

Assume that $q'_0,\cdots, q'_{p-1}$ and $N_0, \cdots, N_{p-1}$ are already constructed. 
First we choose an integer $N_p$ so that (i) holds, and $e^{2 \pi i N_{p}\theta}$ is close enough to $1$ in order to satisfy 
property (ii). 
Then we choose $q'_p$ according to property (iii), while large enough to guarantee property (iv). 
The latter choice is possible due to our assumption on the growth of $(q_n)_{n\geq 0}$, that is, the radius of convergence of the 
series $\sum_{n\geq 0} z^{q_n}/ q_{n+1}$ is infinity.

Let us denote by $p'_n$ the integer corresponding to $q'_n$ in the best rational approximation of $\theta$. 
We have $\vert 2\pi q'_n N_p \theta -2\pi N_p p'_n\vert \le 2\pi N_p/ q'_{n+1}$. 
On the other hand, by property (iii) in the above list, for all $n\geq p$, we have $N_p/ q'_{n+1} \leq N_p/ q'_p \leq 1/2$. 
These imply that for all $n\geq p$, we must have 
\begin{eqnarray}\label {Est1}
\vert 1-e^{2i\pi q'_{n} N_{p}\theta} \vert \leq 2\pi N_p/ q'_{n+1}.
\end{eqnarray}

Also note that since $(q'_n)_{n\geq 0}$ is a subsequence of $(q_n)_{n\geq 0}$, by property (iv) in the above list we have  
\[\sum_{n\ge p} \frac{p^{q'_n}}{q'_{n+1}} \le \sum_{q_n\ge q'_{p}} \frac{p^{q_n}}{q_{n+1}} \leq \frac{\varepsilon_p}{4\pi N_p}.\]
 
We may now check that the conclusion of the lemma is fulfilled. Let $p\ge 1$. Using the estimate in property (ii), and 
the above estimates, we get 
\begin{multline*}
\Big |1-e^{2\pi i N_{p}\theta} |  + \sum_{n= 0}^{p-1} \Big | 1-e^{2\pi i q'_{n} N_{p}\theta} \Big | p^{q'_n} 
 + \sum_{n\ge p} \Big | 1-e^{2 \pi i q'_{n} N_{p}\theta} \Big | p^{q'_n} \\
\le \varepsilon_p/2 + 2\pi N_p\sum_{n\ge p} \frac{p^{q'_n}}{q'_{n+1}} 
\leq \varepsilon_p/2 + \varepsilon_p/2. \qedhere
\end{multline*}
\end{proof}

Assume that $\theta$ is an irrational number which satisfies the hypothesis of Proposition~\ref{P:growth-extra}, so that for  
any subsequence $q'$ of $q$ and any $u \in l^\infty$, $\varphi_{\theta, q', u}$ is an entire function. 
It follows from the formula in Equation~\eqref{E:composition} that for any integer $N\in \mathbb{N}$, $\varphi_{N\theta,q',u}$ is an entire 
function. 

\begin{lemma}\label{LemJul}
Assume that $\theta$ is an irrational number whose best rational approximants $(p_n/q_n)_{n\geq 0}$ satisfy 
$q_n^{-1} \log q_{n+1} \to \infty$, as $n\to \infty$. 
For any subsequence $q'$ of $(q_n)_{n\geq 0}$ and any $z\in \mathbb{C}$ with $|z|>1$, there exists a dense $G_\delta$-subset $E(z)$ 
of $l^\infty$ such that for every $u\in E(z)$, 
\[\sup_N \vert \varphi'_{N\theta,q',u}(z)\vert =+\infty.\]
\end{lemma}

\begin{proof} 
Let us fix an arbitrary subsequence $q'$ and a point $z\in \mathbb{C}$ with $|z|>1$. 
Recall that 
\[\varphi_{N\theta,q',u} (z)=ze^{2\pi iN\theta}\sum_{n=0}^\infty u_{q'_{n}}(1-e^{2\pi i q'_n N\theta}) z^{q'_n}.\]
For each $N\in \mathbb{N}$, we consider the bounded linear operator $\Phi_{N,z}: l^\infty \to \mathbb{C}$, defined as 
\[\Phi_{N,z}(u):= \varphi'_{N\theta,q',u}(z)
=e^{2\pi iN\theta}\sum_{n=0}^\infty u_{q'_{n}}(1+q'_n)(1-e^{2\pi i q'_n N\theta}) z^{q'_n}.\]
The operator norm of $\Phi_{N,z}$ satisfies 
\[\Vert \Phi_{N,z}\Vert = \sum_{n=0}^\infty (1+q'_n) \vert 1-e^{2\pi i q'_n N\theta}\vert \vert z\vert^{q'_n}.\]
Recall that $\varphi_{N\theta,q'}$ is the entire function corresponding to the constant sequence $u_n=1$, and observe that
\begin{eqnarray}\label{Obs}
\vert \varphi'_{N\theta,q'}(z) \vert \le \Vert \Phi_{N,z}\Vert=\Vert \Phi_{N,\vert z\vert}\Vert.
\end{eqnarray}

Let us consider the quantity $K_z:= \sup_N \Vert \Phi_{N,z}\Vert$. 
We claim that for every $z\in \mathbb{C}$ with $|z| > 1$, $K_z=+\infty$. 
Otherwise, by Equation~\eqref{Obs} and the maximum modulus principle, we must have 
$\vert \varphi'_{N\theta,q'}(z')\vert \le K_0$, for every integer $N$ and every $z' \in \mathbb{C}$ with $\vert z' \vert \le \vert z\vert$. 
This implies that the family of maps $\{\varphi_{N\theta,q'} ; N\geq 0\}$ is uniformly bounded on the disc $\Delta(0,(1+|z|)/2)$. 
Then, since $A_{\theta,q'}^{\circ N} (w,z)= (e^{2\pi i N\theta} w+ \varphi_{N\theta,q'}(z), e^{2\pi i N\theta}z)$,
the iterates $(A_{\theta,q'}^{\circ N})_{N\geq 0}$ must be uniformly bounded on the tube $\mathbb{C}\times \Delta(0,(1+|z|)/2)$.
This contradicts the latter part of Theorem~\ref{T:Siegel-component}.

By the above paragraph, for $z\in \mathbb{C}$ with $|z|>1$ we have $\sup_N \Vert \Phi_{N,z}\Vert= +\infty$. 
We may employ the Banach-Steinhaus theorem \cite[page 98]{Rud87} to the family of linear operators $\{\Phi_{N,z}\}_{N\geq 1}$. 
That gives us a a dense $G_\delta$-subset $E(z)$ of $l^\infty$, such that for every $u\in E(z)$, 
$\sup_N \vert \varphi'_{N\theta,q',u}(z)\vert=\sup_N \vert \Phi_{N,z}(u)\vert =+\infty$.
\end{proof}

\begin{proof}[Proof of Theorem \ref{T:Julia}]
Let $(\varepsilon_p)_{p\geq 1}$ be a sequence of positive numbers converging to zero. 
With Lemma~\ref{LemRec}, we generate a subsequence $q'$ of $(q_n)_{n\geq 0}$ and an increasing sequence of integers 
$(N_p)_{p\geq 1}$. 
We also employ Theorem \ref{T:Siegel-component} with $q'$ to obtain the dense $G_\delta$ subset of $l^\infty$ denoted by $E_1$. 
Similarly, for $q'$ and each $z\in S$, we use Lemma \ref{LemJul} to obtain the dense $G_\delta$ set $E(z)$ in $l^\infty$.  
By Baire's theorem, $E_2:= \cap_{z\in S} E(z)\cap E_1$ is a dense $G_\delta$-subset of $l^\infty$.

Fix an arbitrary $u\in E_2$. 
According to Theorem \ref{T:Siegel-component}, the map $A_{\theta,q',u}$ is analytically conjugate to the rotation 
$R_\theta$ on $\mathbb{C}\times \Delta$. Moreover, the orbit of  any $(w,z) \in \mathbb{C}^2$ with $|z|>1$ by the map 
$A_{\theta,q',u}$ is unbounded.

Let $(N_p)_{p\geq 1}$ be the sequence of integers given by Lemma \ref{LemRec}.
Recall that
\begin{align*}
A_{\theta,q',u}^{\circ N_p}(w,z) =(e^{2\pi i N_p\theta}w+ \varphi_{N_p\theta,q',u}(z), e^{2\pi i N_p\theta}z),
\end{align*}
where
\[\varphi_{N_p\theta,q',u} (z)=ze^{2\pi iN_p\theta}\sum_{n=0}^\infty u_{q'_{n}}(1-e^{2\pi i q'_n N_p\theta}) z^{q'_n}.\]
Fix an arbitrary $(w,z)\in\mathbb{C}^2$. For integers $p \geq |z|$, we have  
\begin{align*}
\Big \Vert A_{\theta,q',u}^{\circ N_p}(w,z) &  -(w,z) \Big \Vert \\
& \leq \Vert (w,z)\Vert \cdot \Big\vert 1- e^{2\pi i N_p\theta}\Big \vert 
+ |z| \; \Vert u\Vert_\infty  \sum_{n\ge 0} \vert 1-e^{2\pi i q'_{n} N_{p}\theta} \vert \; |z|^{q'_n}\\
& \leq \Vert (w,z)\Vert \cdot \Big\vert 1- e^{2\pi i N_p\theta}\Big \vert 
+ |z| \Vert u\Vert_\infty  \sum_{n\ge 0} \vert 1-e^{2\pi i q'_{n} N_{p}\theta} \vert \;  p^{q'_n}\\
& \leq \Big (\Vert (w,z)\Vert+ |z| \Vert u\Vert_\infty\Big) \Big(\vert 1- e^{2\pi i N_p\theta}\vert  
+  \sum_{n\ge 0} \vert 1-e^{2\pi i q'_{n} N_{p}\theta} \vert  p^{q'_n}\Big) \\
& \leq  \Big (\Vert (w,z)\Vert+ |z| \Vert u\Vert_\infty\Big) \varepsilon_p.
\end{align*}
In the last line of the above equation we have used the estimate in Lemma \ref{LemRec}. 
The above bound shows that the orbit of $(w,z)$ by $A_{\theta,q',u}$ is recurent.

Evidently, $\Vert D A_{\theta,q',u}^N (w,z)\Vert \ge |\varphi'_{N\theta,q',u}(z)|$. 
Therefore, by Lemma \ref{LemJul}, for every $(w,z) \in S$, we have $\sup_N \Vert D A_{\theta,q',u}^N (w,z)\Vert =+\infty$. 
\end{proof}

\subsubsection*{Acknowledgements} 
Both authors would like to thank CNRS-Imperial ``Abraham de Moivre'' International Research Laboratory for their kind support 
and hospitality during the visit of Berteloot to Imperial College in Summer 2019, when this research was carried out. 
Cheraghi is grateful to EPSRC(UK) for grant No. EP/M01746X/1 - rigidity and small divisors in holomorphic dynamics. 

\bibliographystyle{amsalpha}
\bibliography{Data}
\end{document}